\documentclass{article} 

\usepackage{amsfonts,amssymb,amsmath,amsthm,epsfig} 
\usepackage[all, poly]{xy} 

\newlength{\CutOffSpace}
\title{Coxeter elements of the symmetric groups whose powers afford the longest}
\author{Masashi KOSUDA} 

\theoremstyle{plain}
\newtheorem{thm}{Theorm}[section]
\newtheorem{cor}[thm]{Corollary}
\newtheorem{prop}[thm]{Proposition}
\newtheorem{lem}[thm]{Lemma}
\theoremstyle{definition}
\newtheorem{defn}[thm]{\sc Definition}
\newtheorem{rem}[thm]{\sc Remark}

\begin{document}
\maketitle

\baselineskip=1.0\baselineskip

\section*{Introduction}
It is well known that the symmetric group $\mathfrak{S}_{n}$
is defined by the generators $S = \{s_i\}_{i=1}^{n-1}$
of transpositions.
Consider a product of the distinct $n-1$ generators in any order
$s_{i_1}s_{i_2}\cdots s_{i_{n-1}}$.
Such an element is called a {\em Coxeter element}.
All Coxeter elements are conjugate to each other and
have the same cycle type $(n)$,
one cycle of length $n$ (Remark~\ref{rem:cyclic}),
and accordingly $(s_{i_1}s_{i_2}\cdots s_{i_{n-1}})^n = 1$.

Now suppose that $n$ is even
and consider the power of a Coxeter element
$(s_{i_1}s_{i_2}\cdots s_{i_{n-1}})^{n/2}$.
If this is reduced,
then it is the longest,
since the longest word of $\mathfrak{S}_{n}$ is the unique
element of length $n(n-1)/2$.
A natural question arise:
Does $(s_{i_1}s_{i_2}\cdots s_{i_{n-1}})^{n/2}$
afford the longest word in $\mathfrak{S}_{n}$
for any permutation of the generators?
Actually, this does not hold.
For example, $(s_1s_3s_2)^2$ affords the longest in $\mathfrak{S}_4$
while $(s_1s_2s_3)^2$ does not.

In this article, we first show that in case $n$ is even
which Coxeter element in $\mathfrak{S}_{n}$
affords the longest by taking its power to $n/2$.

Then we also consider the case where $n$ is odd, say $n=2m-1$.
In this case we cannot define
$(s_{i_1}s_{i_2}\cdots s_{i_{2m-2}})^{n/2}$.
Instead, we consider the following word
\[
	w_2(w_1w_2)^{m-1},
\]
where $w_1 = s_{i_1}s_{i_2}\cdots s_{i_{m-1}}$
and $w_2 = s_{i_m}s_{i_{m+1}}\cdots s_{i_{2m-2}}$.
For some Coxeter elements,
this expression also affords the longest word in $\mathfrak{S}_{2m-1}$.
For example, a Coxeter element $s_1s_3s_2s_4$ affords the longest word
$s_2s_4(s_1s_3s_2s_4)^2$ in $\mathfrak{S}_5$,
while $s_3s_4(s_1s_2s_3s_4)^2$ does not.
We also show that in case $n$ is odd
which Coxeter element
affords the longest in $\mathfrak{S}_{n}$.

\section{Preliminaries}
The symmetric group $\mathfrak{S}_{n}$
is a Coxeter system of type $A_{n-1}$~\cite{Bou,H},
which is defined by the generators:
\begin{equation*}
	S = \{s_1, s_2, \ldots, s_{n-1}\}
\end{equation*}
and the relations:
\begin{eqnarray*}
	s_i^2 = 1& &(1\leq i\leq n-1),\\
	s_is_{i+1}s_i = s_{i+1}s_is_{i+1}& &(1\leq i\leq n-2),\\
	s_is_j = s_js_i & &(1\leq i, j\leq n-1,\ |i-j|\geq 2).
\end{eqnarray*}
Each $w\in\mathfrak{S}_{n}$ can be written in the form
$w = s_{i_1}s_{i_2}\cdots s_{i_r}$ (not necessarily distinct)
for $s_{i_j}\in S$ ($j = 1, 2, \ldots, r$).
If $r$ is as small as possible,
then call it the {\em length} of $w$ and written $\ell(w)$,
and call any expression of $w$ as a product of $r$ elements of $S$
a {\em reduced expression}.
There may be some reduced expressions for an element $w\in\mathfrak{S}_{n}$.

Throughout this paper,
we will describe elements of $\mathfrak{S}_{n}$
drawing the following pictures called the Amida diagrams.
\begin{eqnarray}\label{xy:Amida3231}
\begin{xy}
(0,-5)="B1",(0,5)="T1",
(5,-5)="B2",(5,5)="T2",
(10,-5)="B3",(10,5)="T3",
(15,-5)="B4",(15,5)="T4",
{\ar @{-}"B1";"T1"},
{\ar @{-}"B2";"T2"},
{\ar @{-}"B3";"T3"},
{\ar @{-}"B4";"T4"},
{\ar @{-}(0,-3);(5,-3)},
{\ar @{-}(5,1);(10,1)},
{\ar @{-}(10,-1);(15,-1)},
{\ar @{-}(10,3);(15,3)},
\end{xy}\quad .
\end{eqnarray}

An {\em Amida diagram}
consists of $n$ vertical lines and horizontal segments
placed between the adjacent vertical lines like ladders
so that the end points of each horizontal segment meet the vertical lines
and that they do not meet any other horizontal segments' end points.

The $n$ runners who start from the bottoms of the vertical lines
go up along the lines.
If they find horizontal segments on their right [resp.~left],
they turn right [resp.~left] and go along the segments.
They necessarily meet the adjacent vertical lines.
Then again go up the vertical lines and iterate this trip
until they arrive at the tops of the vertical lines.
If the $i$-th  runner arrives at
the $\sigma_i$-th position $(i=1,2,\ldots,n)$,
we consider the Amida diagram is one of the expressions of
\begin{eqnarray*}
	\sigma =\left(
		\begin{array}{cccccc}
		1&2&\cdots&i&\cdots&n\\
		\sigma_1&\sigma_2&\cdots&\sigma_i&\cdots&\sigma_{n}
		\end{array}
	\right)\in\mathfrak{S}_{n}.
\end{eqnarray*}
For example the Amida diagram as in \eqref{xy:Amida3231} corresponds to
$\displaystyle{\left(
	\begin{array}{cccc}1&2&3&4\\4&1&3&2\end{array}
\right)}$.

We can also consider the product of Amida diagrams.
If $D_1$ and $D_2$ are Amida diagrams of $\mathfrak{S}_{n}$
then the product $D_1D_2$
is defined to be an Amida diagram obtained from
$D_1$ and $D_2$ by putting the former on the latter.

A generator $s_i\in\mathfrak{S}_{n}$ corresponds to
an Amida diagram which consists of $n$ vertical lines
with only one horizontal segment between the $i$-th and the $(i+1)$-st
vertical lines.
\begin{equation}\label{xy:si}
\begin{xy}
(0,-5)="B1",(0,5)="T1",
(15,-5)="B2",(15,5)="T2",
(20,-5)="B3",(20,5)="T3",
(25,-5)="B4",(25,5)="T4",
(30,-5)="B5",(30,5)="T5",
(45,-5)="B6",(45,5)="T6",
{\ar @{}(1,-7);(1,7)^>1^<1},
{\ar @{}(21,-7);(21,7)^>i^<i},
{\ar @{}(29,-7);(29,7)^>{i+1}^<{i+1}},
{\ar @{}(48,-7);(48,7)^>{n}^<{n}},
{\ar @{-}"B1";"T1"},
{\ar @{-}"B2";"T2"},
{\ar @{-}"B3";"T3"},
{\ar @{-}"B4";"T4"},
{\ar @{-}"B5";"T5"},
{\ar @{-}"B6";"T6"},
{(5, 0) \ar @{{.}{.}}(10, 0)},
{(35, 0) \ar @{{.}{.}}(40, 0)},
{\ar @{-}(20,0);(25,0)},
\end{xy}\quad .
\end{equation}
Since $s_i$s are the generators,
any word in $\mathfrak{S}_{n}$
can be expressed as an Amida diagram.
For example \eqref{xy:Amida3231} denotes $s_3s_2s_3s_1$.

A {\em Coxeter element} in $\mathfrak{S}_{n}$
is a product of distinct $n-1$ generators
$\{s_1, \ldots, s_{n-1}\}$ in any order.
By the definition we have $(n-1)!$ expressions of length $n-1$
for all Coxeter elements.
However it may happen that
different permutations of the distinct $n-1$ generators
cause the same Coxeter element.
For example, expressions $s_1s_3s_2s_4$
and $s_3s_1s_4s_2$ are the same element.

We want to count all distinct Coxeter elements.
Amida diagrams give us a convenient tool for doing it.
In order to express a Coxeter element by an Amida diagram,
we have only to place the $i$-th horizontal segment
(which corresponds to $s_i$)
between the $i$-th and the $(i+1)$-st vertical lines:
Place the first segment between the 1-st and the 2-nd vertical lines.
The second segment is placed between the 2-nd and the 3-rd vertital lines
so that it does not be placed on the same height as the 1-st one's.
Iterate this procedure until $(n-1)$-st segment is placed.
We call the Amida diagrams obtained in this way
{\em standard}.
The following is a standard Amida diagram for a Coxeter element
$s_1s_2s_4s_3s_5$ in $\mathfrak{S}_{6}$.
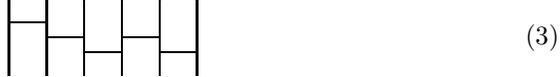
\begin{equation}
\begin{xy}
(0,-5)="B1",(0,5)="T1",
(5,-5)="B2",(5,5)="T2",
(10,-5)="B3",(10,5)="T3",
(15,-5)="B4",(15,5)="T4",
(20,-5)="B5",(20,5)="T5",
(25,-5)="B6",(25,5)="T6",
{\ar @{-}"B1";"T1"},
{\ar @{-}"B2";"T2"},
{\ar @{-}"B3";"T3"},
{\ar @{-}"B4";"T4"},
{\ar @{-}"B5";"T5"},
{\ar @{-}"B6";"T6"},
{\ar @{-}(0,2);(5,2)},
{\ar @{-}(5,0);(10,0)},
{\ar @{-}(10,-2);(15,-2)},
{\ar @{-}(15,0);(20,0)},
{\ar @{-}(20,-2);(25,-2)},
\end{xy}
\end{equation}

Since the $i$-th segment is not placed on the same height
as the $(i-1)$-st one's,
the former one must be placed higher or lower than the latter one's.
For the fixed $n$ vertical lines,
the standard Amida diagram for an expression of
a Coxeter element is uniquely defined up to graph isotopy.
This graph isotopy also compatible with
the commutativity among non-adjacent generators in $S$
of $\mathfrak{S}_n$.
 
We label  the $i$-th vertical line with $+$ or $-$ sign
($i=2,3,\ldots n-1$),
according to the positions of the $(i-1)$-st and the $i$-th
horizontal segments.
If the $i$-th horizontal segment is
placed higher [resp.~lower] than the $(i-1)$-st one's
we label  the $i$-th vertical line with $+$ [resp.~$-$] sign.
Then we have a sequence $[\epsilon_2, \epsilon_3, \ldots, \epsilon_{n-1}]$
of $+$ and $-$ signs of length $n-2$.
Conversely,
from a sequence of $+$ and $-$ signs of length $n-2$,
we can obtain the corresponding Coxeter element:
if $\epsilon_2$ is positive [resp.~negative]
then multiply $s_1$ by $s_2$ from the left [resp.~right].
if $\epsilon_3$ is positive [resp.~negative]
then multiply the previous one by $s_3$ from the left [resp.~right]
and repeat these multiplications until $s_{n-1}$ is multiplied.
Thus we can expect that the  following theorem holds.
\begin{thm}
There are $2^{n-2}$ Coxeter elements in $\mathfrak{S}_{n}$.
\end{thm}
To prove the theorem above,
we have only to show that distinct sequences of signatures
give distinct elements in $\mathfrak{S}_n$.
This will be shown after Remark~\ref{rem:cyclic}.

For $C$ a Coxeter element in $\mathfrak{S}_{n}$,
let $\epsilon = [\epsilon_2, \epsilon_3, \ldots, \epsilon_{n-1}]$
be a sequence of plus and minus signs of length $n-2$
defined above (each sign is tagged to vertical lines
except the left most and the right most ones).
We call $\epsilon$ a {\em Coxeter path} of $C$ and denote it by $p(C)$.
Using the Coxeter path $p(C)$,
we can define the {\em height} $ht(C)$ of $C$
by $ht(C)=\sum_{i=2}^{n-1}\epsilon_i$.
Here $\epsilon_i$ takes the value $+1$ [resp.~$-1$]
if $+$ [resp.~$-$] sign is assigned.

We also introduce the notion of {\em stanza} and {\em co-stanza}
of a standard Amida diagram of a Coxeter element.
Stanzas are ascending staircases from lower left to upper right
and co-stanzas are ascending staircases from lower right to
upper left.
We label the beginning points of
stanzas [resp.~co-stanzas] as $p_1, p_2, \ldots$
[resp.~$q_1, q_2, \ldots$] from left to right
[resp.~right to left]
and call the stanza [resp.~co-stanza]
which starts at $p_i$ [resp.~$q_i$] the $i$-th stanza
[resp.~$i$-th co-stanza].
For example, the Amida diagram~\eqref{xy:stanzas} of heights 1 
has 4 stanzas and 5 co-stanzas
which start at $p_1, p_2, p_3, p_4$
and $q_1, q_2, q_3, q_4, q_5$ respectively.
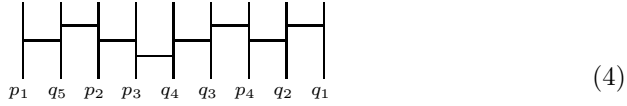
\begin{equation}\label{xy:stanzas}
\begin{xy}
(0,0)="B1",(0,10)="T1",
(5,0)="B2",(5,10)="T2",
(10,0)="B3",(10,10)="T3",
(15,0)="B4",(15,10)="T4",
(20,0)="B5",(20,10)="T5",
(25,0)="B6",(25,10)="T6",
(30,0)="B7",(30,10)="T7",
(35,0)="B8",(35,10)="T8",
(40,0)="B9",(40,10)="T9",
{\ar @{-}"B1";"T1"},
{\ar @{-}"B2";"T2"},
{\ar @{-}"B3";"T3"},
{\ar @{-}"B4";"T4"},
{\ar @{-}"B5";"T5"},
{\ar @{-}"B6";"T6"},
{\ar @{-}"B7";"T7"},
{\ar @{-}"B8";"T8"},
{\ar @{-}"B9";"T9"},
{\ar @{-}(0,5);(5,5)}, 
{\ar @{-}(5,7);(10,7)},
{\ar @{-}(10,5);(15,5)},
{\ar @{-}(15,3);(20,3)},
{\ar @{-}(20,5);(25,5)},
{\ar @{-}(25,7);(30,7)},
{\ar @{-}(30,5);(35,5)},
{\ar @{-}(35,7);(40,7)},
{\ar @{}(2,-2);(2,10)^<{p_1}},
{\ar @{}(7,-2);(7,10)^<{q_5}},
{\ar @{}(12,-2);(12,10)^<{p_2}},
{\ar @{}(17,-2);(17,10)^<{p_3}},
{\ar @{}(22,-2);(22,10)^<{q_4}},
{\ar @{}(27,-2);(27,10)^<{q_3}},
{\ar @{}(32,-2);(32,10)^<{p_4}},
{\ar @{}(37,-2);(37,10)^<{q_2}},
{\ar @{}(42,-2);(42,10)^<{q_1}},
\end{xy}
\end{equation}
As for the stanzas and co-stanzaz, we note the following.
\begin{rem}\label{rem:cyclic}
Let $p_1=1, p_2, \ldots, p_s$ and $q_1=n, q_2,\ldots, q_t$ be
the beginning points of stanzas and co-stanzas respectively
of the standard Amida diagram of a Coxeter element $C$.
\begin{enumerate}
\item[(1)]
Let $p(C)$ be the Coxeter path of $C$.
Then $p_2, p_3, \ldots$
[resp.~$q_2, q_3, \ldots$] correspond to the coordinates of $p(C)$
which have $-$ [resp.~$+$] signs.
\item[(2)]
The height $ht(C)$ of $C$ is equal to
the number of co-stanzas minus the number of stanzas.
Namely $ht(C) = t - s$, which is also equal to the number
of $+$ signs minus the number of $-$ signs
in $p(C)$.
\item[(3)]
$C(p_1) = p_2, \ldots, C(p_{s-1}) = p_s, C(p_s) = q_1$,
$C(q_1) = q_2, \ldots, C(q_{t-1}) = q_t, C(q_t) = p_1$.
In particular, all Coxeter elements are conjugate\footnote{
In any type of finite irreducible Coxeter groups,
all Coxeter elements are conjugate\cite{Bou,H}.}
and their cycle type is $(n)$ (1-cycle of length $n$).
\end{enumerate}
\end{rem}
\begin{proof}[{\sc Proof of Theorem~1.1}]
As we stated before Theorem~1.1,
there exists one to one correspondence between
the standard Amida diagrams for Coxeter elements in $\mathfrak{S}_n$
and sequences of $+$ and $-$ signs of length $n-2$.
The sequences of the signs determine
the beginning points of stanzas and co-stanzas uniquely.
By Remark~\ref{rem:cyclic}(3),
$C$, a Coxeter element in $\mathfrak{S}_n$
which (by the Amida diagram)
causes stanzas $p_1, p_2, \ldots$ and co-stanzas
$q_1, q_2, \ldots$ 
is a cyclic permutation $(p_1=1, p_2, \ldots, q_1, q_2, \ldots)$.
This explains that Coxeter elements determined by
the sequences of signs are all distinct.
Hence we find
the number of Coxeter elements in $\mathfrak{S}_n$ is $2^{n-2}$.
\end{proof}

For $\sigma$ an element of $\mathfrak{S}_{n}$,
the {\em inversion number} $\iota(\sigma)$ is defined by
\begin{eqnarray*}
	\iota(\sigma) &=& |\{(i,j)\ ;\ i< j,\ \sigma(i) > \sigma(j)\}|\ .
\end{eqnarray*}
The inversion number $\iota(\sigma)$
coincides with the length $\ell(\sigma)$
and there is an Amida diagram for $\sigma$
which has $\ell(\sigma)$ horizontal segments.
The longest word $w_0\in\mathfrak{S}_{n}$ maps $i$ to $w_0(i) = n+1 -i$,
and $\ell(w_0) = n(n-1)/2$.

In terms of Amida diagrams, it is easy to show that
Coxeter elements are characterized by the cycle type and the inversion number.
\begin{prop}
Let $\sigma$ be an element of  $\mathfrak{S}_{n}$.
If $\iota(\sigma) = n-1$ and the cycle type of  $\sigma$ is $(n)$
1-cycle of length $n$,
then $\sigma$ is a Coxeter element of  $\mathfrak{S}_{n}$.
\end{prop}
\begin{proof}[{\sc Proof}]
Since $\iota(\sigma) = \ell(\sigma)$,
we have an expression of $\sigma$ whose Amida diagram
has exactly $n-1$ horizontal segments.
If there are more than one horizontal segments between
adjacent pair of vertical lines,
then there is an adjacent pair of vertical lines
which have no horizontal segments between them.
A runner who starts at one of the bottom of them cannot move to
the other. Such an Amida diagram is not a cycle of length $n$.
Thus, if $\iota(\sigma) = n-1$
and the cycle type of $\sigma$ is $(n)$,
then its Amida diagram consists of $n-1$ horizontal segments,
each of which is assigned each adjacent pair of the vertical lines.
This implies $\sigma$ is a Coxeter element in $\mathfrak{S}_n$.
\end{proof}

\section{Coxeter elements which afford the longest word in $\mathfrak{S}_{2m}$}

In the previous section,
we defined the standard Amida diagrams of Coxeter elements
and showed that
each of them is distinguished by a sequence of plus and minus signs
of length $n-2$.
In this section, we characterize the Coxeter elements which afford
the longest element,
when $n-1$, the number of Coxeter generators, is odd, say $n=2m$.

Let $C$ be a Coxeter element in
$\mathfrak{S}_{2m} = \langle s_1, s_2, \ldots, s_{2m-1}\rangle$.
Recall that the Coxeter number $h$ (which is equal to the order of $C$)
is $n=2m$ (Remark~\ref{rem:cyclic}(3)).
In order that $C^{h/2} = C^{n/2} = C^m$ is the longest,
it should hold that
\begin{equation}\label{eq:Cpow}
	C^{m}(j) = n+1-j  \mbox{ for } j= 1, 2, \ldots, n.
\end{equation}
Since $C$ is a bijection, there exists $k\in\{1, 2, \ldots, n\}$
such that $C(k) = j$.
Hence \eqref{eq:Cpow} would be written as $C^{m}(C(k)) = n+ 1- C(k)$
for $k = 1, 2, \ldots, n$.
Again applying \eqref{eq:Cpow} for $k$, we obtain
\begin{equation*}
	C^m (k) = n+1-k
\end{equation*}
and we have
\begin{equation*}
	C^{m}(C(k)) = C(C^m(k)) = C(n+1-k) = n + 1 -C(k).
\end{equation*}
This implies that
in the $m$-th power of the standard Amida diagram of $C$,
a runner who starts $k$-th position from the left
arrives at the $C(k)$-th position from the left,
while a runner who starts
at the $k$-th position from the right
arrives at the $C(k)$-th position from the right.
Hence we have the following theorem.

\begin{thm}
Let $n = 2m$ be an even integer
and $C$ a Coxeter element in $\mathfrak{S}_{2m}$.
Then $C^m$ is the longest word in $\mathfrak{S}_{2m}$
if and only if the corresponding standard Amida diagram of $C$
is symmetric with respect to the vertical axis
between the $m$-th and the $(m+1)$-st vertical lines.
\end{thm}

The symmetric standard Amida diagrams as in the theorem above
are obtained from the left half of the diagram by reflecting 
the image of it
with respect to the vertical axis.
Hence we have the following Corollary.

\begin{cor}\label{cor:numofcoxodd}
The number of distinct
Coxeter elements which satisfy the above theorem is $2^{m-1}$.
\end{cor}

\section{Admissible Coxeter elements}
Before we consider the case $n=2m-1$,
we introduce the notion of admissible Coxeter elements.
Admissible Coxeter elements are inductively defined
from the ones in $\mathfrak{S}_{n}$
to the ones in $\mathfrak{S}_{n+2}$.

Let us consider the symmetric group $\mathfrak{S}_{n+2}$
as the permutation group of $n+2$ letters
$\{0, 1, 2, \ldots, n, n+1\}$
generated by the transpositions $s_0=(0,1)$, $s_1 = (1,2)$, $s_2 = (2,3)$,
$\ldots$, $s_{n-1} = (n-1, n)$ and $s_{n} = (n, n+1)$.
For
\[
	w = \begin{pmatrix}
		1&2&\cdots &n\\
		i_1&i_2&\cdots &i_{n}
	\end{pmatrix}
	\in\mathfrak{S}_{n}
\]
we define $\overline{w}\in \mathfrak{S}_{n+2}$ by
\[
	\overline{w} =
	\begin{pmatrix}
		0&1&2&\cdots &n&n+1\\
		0&i_1&i_2&\cdots &i_{n}&n+1
	\end{pmatrix}.
\]
If there is no confusion,
we merely write $w$ to refer the image $\overline{w}\in\mathfrak{S}_{n+2}$.

\begin{defn}
Let $C$ be a Coxeter element in $\mathfrak{S}_{n}$.
We identify $C$ with $\overline{C}\in\mathfrak{S}_{n+2}$.
Then $s_0s_{n}C$, $Cs_0s_{n}$, $s_0Cs_{n}$ and $s_{n}Cs_0$
are all Coxeter elements in $\mathfrak{S}_{n+2}$.
We call these elements {\em extensions} of $C$.
\end{defn}
\begin{rem}
Every Coxeter element in $\mathfrak{S}_{n+2}$ is obtained
from $\mathfrak{S}_{n}$ above way.
In other words, a Coxeter element $C\in\mathfrak{S}_{n}$
has (exactly) one of the expression of the form
$s_1s_{n-1}C'$, $C's_1s_{n-1}$, $s_1C's_{n-1}$ or $s_{n-1}C's_1$,
where $C'$ is an expression of a Coxeter element in
$\mathfrak{S}_{n-2} = \langle s_2, s_3, \ldots, s_{n-2}\rangle$.
\end{rem}

As for the heights of extensions, we have the following lemma.
\begin{lem}\label{lem:admissible}
Let $C$ be a Coxeter element in $\mathfrak{S}_{n}$
and $\eta =ht(C)$ its height.
Then  $ht(s_0s_{n}C)$, $ht(Cs_0s_{n})$, $ht(s_0Cs_{n})$
and $ht(s_{n}Cs_0)$
are $\eta$, $\eta$, $\eta-2$ and $\eta+2$ respectively.
\end{lem}

\begin{proof}[{\sc Proof}]
In terms of the standard Amida diagrams,
multiplying $s_0$ from the left corresponds to adding a horizontal segment
in a higher  position to the $s_1$'s segment.
Since the height of a Coxeter element is measured by the relative position
of the right most horizontal segment with respect to the left most horizontal segment,
this addition of $s_0$ lowers the height of $C$ by 1.
Similarly, multiplying $s_0$ from the right raises the height of $C$ by 1.
Multiplying $s_{n}$ from the left [right]
also raises [lowers] the height of $C$ by 1.
The result follows from these observations.
\end{proof}

Under the preparation above, 
the admissible Coxeter elements in $\mathfrak{S}_{2m-1}$
are defined as follows.

\begin{defn}\label{defn:admissible}
There exists two Coxeter elements $s_1s_2$ and $s_2s_1$
in $\mathfrak{S}_{3}$.
Both of them are by definition admissible.
Let $C\in\mathfrak{S}_{2m-1} (m\geq 2)$ be an admissible Coxeter element
and ${\cal E}(C)$ one of the extensions of $C$.
If $|ht({\cal E}(C))|\leq 1$ then the extension is called {\em admissible}.
Otherwise the extension is {\em non-admissible}.
An {\em admissible Coxeter elements} in $\mathfrak{S}_{2m-1}$
is defined as a Coxeter element
in $\mathfrak{S}_{2m-1}$ obtained from
$s_1s_2$ or $s_2s_1$ in $\mathfrak{S}_{3}$ by the iterative application of
the admissible extensions.
\end{defn}

Since the height of $s_2s_1$ [resp.~$s_1s_2$] in $\mathfrak{S}_3$
is $1$ [resp.~$-1$],
by Lemma~\ref{lem:admissible} the heights of admmisible Coxeter
elements are $+1$ or $-1$.
So the definition above is rewritten as follows.

\begin{rem}\label{rem:admissible}
Let $C$ be an admissible Coxeter element in
$\mathfrak{S}_{2m-1}$
\begin{enumerate}
\item[(1)]
If $C$ has an expression such that $ht(C)=1$,
then the expressions $s_0s_{2m-1}C$, $Cs_0s_{2m-1}$ and $s_0Cs_{2m-1}$ are
admissible in $\mathfrak{S}_{2m+1}$.
\item[(2)]
If $C$ has an expression such that $ht(C)=-1$,
then the expressions $s_0s_{2m-1}C$, $Cs_0s_{2m-1}$ and $s_{2m-1}Cs_0$ are
admissible in $\mathfrak{S}_{2m+1}$.
\end{enumerate}
\end{rem}

Note that $ht(C) = \pm 1$ does not mean $C$ is admissble.
For example $C = s_0\overline{s_4s_3s_2s_1}s_5\in\mathfrak{S}_{5+2}$
has its height $ht(C) = 1$, but $C$ is non-admissible,
since $ht(s_4s_3s_2s_1) = 3$.

From Definition~\ref{defn:admissible} and Remark~\ref{rem:admissible},
we have the following corollary.
\begin{cor}
There are $2\cdot 3^{m-2}$ admissible Coxeter elements
in $\mathfrak{S}_{2m-1}$.
\end{cor}

\section{Coxeter elements which afford the longest word
in $\mathfrak{S}_{2m-1}$}
All Coxeter elements in $\mathfrak{S}_{2m-1}$
have the same order $h = 2m-1$.
In this case,
the situation is rather complicated.
Since $h/2 = (2m-1)/2$ is a half integer,
we cannot define $h/2$-nd power of a Coxeter element.
On the other hand, a Coxeter element $C\in\mathfrak{S}_{2m-1}$
has even length $\ell(C) = 2m-2$.
Hence putting $w_1 = s_{i_1}\cdots s_{i_{m-1}}$
and $w_2 = s_{i_m}\cdots s_{i_{2m-2}}$,
we consider the following word
\begin{equation}\label{eq:hpCox}
	C^{h/2} = C^{h/2}_{w_2} = w_2(w_1w_2)^{m-1}
	= (s_{i_{m}}s_{i_{m+1}}\cdots s_{i_{2m-2}})C^{m-1}.
\end{equation}
Note that the definition of $C^{h/2}_{w_2}$ above
depends on the choice of $w_2$ (and $w_1$).
For example, for $C = C_1 = s_1s_3s_2s_4$
a Coxeter element in $\mathfrak{S}_5$,
another expression $C_2 = s_3s_4s_1s_2$ coincides with $C$.
According to the equation~\eqref{eq:hpCox}, we have
$C_1^{5/2} = s_3s_4C_1^2 = s_3s_4C^2$
and
$C_2^{5/2} = s_1s_2C_2^2 = s_3s_4C^2$
which do not coincide.
However we have the following lemma.
\begin{lem}
Let $C = s_{i_1}s_{i_2}\cdots s_{i_{2m-2}}\in\mathfrak{S}_{2m-1}$
be a Coxeter element in $\mathfrak{S}_{2m-1}$
and $C = w_1w_2$ an expression of $C$
such that $\ell(w_1) = \ell(w_2) = m-1$.
If $C^{h/2}_{w_2}$ affords the longest element in $\mathfrak{S}_{2m-1}$
for the expression, then such $w_2$ is uniquely determined.
\end{lem}
\begin{proof}[{\sc Proof}]
Assume that $w_1w_2$ and $w_1'w_2'$ are both expressions of $C$.
We further assume that both $C^{h/2}_{w_2} = w_2C^{m-1}$
and $C^{h/2}_{w_2'} = w_2'C^{m-1}$ are the longest in $\mathfrak{S}_{2m-1}$.
Since the longest element in $\mathfrak{S}_{2m-1}$ is unique,
they coincide. Hence we have $w_2 = w_2'$.
\end{proof}
By the above lemma,
we merely write $C^{h/2}$ for $C^{h/2}_{w_2}$ in the following lemma.
\begin{lem}
Let $C = s_{i_1}s_{i_2}\cdots s_{i_{2m-2}}\in\mathfrak{S}_{2m-1}$
be an admissible Coxeter element
which affords the longest word in $\mathfrak{S}_{2m-1}$ by $C^{h/2}$.
Let ${\cal E}(C)$ be one of the admissible extensions of $C$.
Then the following holds.
\begin{enumerate}
\item[(1)]
If $ht(C) = 1$ and ${\cal E}(C)$ is written as $s_0s_{2m-1}C$,
then $s_0(s_{i_{m}}s_{i_{m+1}}\cdots s_{i_{2m-2}}){\cal E}(C)^{m}$
is the longest in $\mathfrak{S}_{2m+1}$.
\item[(2)]
If $ht(C) = 1$ and ${\cal E}(C)$ is written as $Cs_0s_{2m-1}$,
then $(s_{i_{m}}s_{i_{m+1}}\cdots s_{i_{2m-2}})s_0{\cal E}(C)^{m}$
is the longest in $\mathfrak{S}_{2m+1}$.
\item[(3)]
If $ht(C) = 1$ and ${\cal E}(C)$ is written as $s_0Cs_{2m-1}$,
then $(s_{i_{m}}s_{i_{m+1}}\cdots s_{i_{2m-2}})s_{2m-1}{\cal E}(C)^{m}$
is the longest in $\mathfrak{S}_{2m+1}$.
\item[(4)]
If $ht(C) = -1$ and ${\cal E}(C)$ is written as $s_0s_{2m-1}C$,
then $s_{2m-1}(s_{i_{m}}s_{i_{m+1}}\cdots s_{i_{2m-2}}){\cal E}(C)^{m}$
is the longest in $\mathfrak{S}_{2m+1}$.
\item[(5)]
If $ht(C) = -1$ and ${\cal E}(C)$ is written as $Cs_0s_{2m-1}$,
then $(s_{i_{m}}s_{i_{m+1}}\cdots s_{i_{2m-2}})s_{2m-1}{\cal E}(C)^{m}$
is the longest in $\mathfrak{S}_{2m+1}$.
\item[(6)]
If $ht(C) = -1$ and ${\cal E}(C)$ is written as $s_{2m-1}Cs_0$,
then $(s_{i_{m}}s_{i_{m+1}}\cdots s_{i_{2m-2}})s_0{\cal E}(C)^{m}$
is the longest in $\mathfrak{S}_{2m+1}$.
\end{enumerate}
\end{lem}

\begin{proof}[{\sc Proof}]
We prove the theorem by induction on $m$.
If $m=2$, then there are two Coxeter elements
$s_2s_1$ and $s_1s_2$ in $\mathfrak{S}_{2\cdot 2 -1}$.
Both of them are by definition admissible
and $s_1(s_2s_1)^1 = s_2(s_1s_2)^1$ is
the longest in $\mathfrak{S}_{2\cdot 2 -1}$.
So both $s_2s_1$ and $s_1s_2$ satisfy the hypothesis.
Consider the case $C=s_2s_1\in\mathfrak{S}_{2\cdot 2 -1}$.
Since $ht(s_2s_1) = 1$ we have only to consider the case (1)(2)(3).
If ${\cal E}(C) = s_0s_3C = s_0s_3s_2s_1$,
then we can check that $s_0(s_1)(s_0s_3s_2s_1)^2$
affords the longest in $\mathfrak{S}_{2\cdot 2+1}$ by direct calculation.
Similarly, if ${\cal E}(C) = Cs_0s_3 = s_2s_1s_0s_3$,
then $(s_1)s_0(s_2s_1s_0s_3)^2$ affords the longest
and if ${\cal E}(C) = s_0Cs_3 = s_0s_2s_1s_3$,
then $(s_1)s_3(s_0s_2s_1s_3)^2$ affords the longest.
The case $C=s_1s_2$ will be verified similarly.

Before moving on to the case $m\geq 3$,
we rewrite the hypothesis.
Let $C = (p_1, p_2, \ldots, p_s, q_1, q_2, \ldots, q_t)$
be the cyclic presentation of the Coxeter element $C$ as in
Remark~\ref{rem:cyclic}(3).
If $ht(C) = -1$,
then taking the mirror image of the Amida diagram of $C$
with respect to the vertical axis,
we can attribute this case to the case $ht(C) = 1$.
Hence we have only to consider the case (1)(2)(3).
In case $ht(C)=1$ the numbers of stanzas and co-stanzas of $C$
are $m-1$ and $m$ respectively.
Hence we put $s = m$ and $t=m+1$.
We note that $p_1 = 1$ and $q_1 = 2m -1$.
The hypothesis that the word $C^{h/2}$ is the longest
implies that $C^{h/2}(j) = 2m-j$.
Hence we have
\begin{eqnarray}
2m-p_j &=& C^{h/2}(p_j)
= s_{i_{m}}s_{i_{m+1}}\cdots s_{i_{2m-2}}C^{m-1}(p_j)\nonumber\\
&=& s_{i_{m}}s_{i_{m+1}}\cdots s_{i_{2m-2}}C^j(p_{m-1})\nonumber\\
&=& s_{i_{m}}s_{i_{m+1}}\cdots s_{i_{2m-2}}C^{j-1}(q_1)\nonumber\\
&=& s_{i_{m}}s_{i_{m+1}}\cdots s_{i_{2m-2}}(q_j)\ (j= 1, 2, \ldots, m-1)
\label{eq:Cpj}
\end{eqnarray}
and
\begin{eqnarray}
2m-q_j &=& C^{h/2}(q_j)
= s_{i_{m}}s_{i_{m+1}}\cdots s_{i_{2m-2}}C^{m-1}(q_j)\nonumber\\
&=& s_{i_{m}}s_{i_{m+1}}\cdots s_{i_{2m-2}}C^{j-1}(q_{m})\nonumber\\
&=& s_{i_{m}}s_{i_{m+1}}\cdots s_{i_{2m-2}}C^{j-2}(p_1)\nonumber\\
&=& s_{i_{m}}s_{i_{m+1}}\cdots s_{i_{2m-2}}(p_{j-1})\ (j= 2, 3, \ldots, m).
\label{eq:Cqj}
\end{eqnarray}
For $q_{1}=2m-1$, we have
\begin{eqnarray}\label{eq:Cq1}
p_1 = 1 &=& 2m-q_{1}\nonumber\\
&=& C^{h/2}(q_{1})
\quad (\mbox{$C^{h/2}$ is the longest in $\mathfrak{S}_{2m-1}$})\nonumber\\
&=& s_{i_{m}}s_{i_{m+1}}\cdots s_{i_{2m-2}}C^{m-1}(q_{1})\nonumber\\
&=& s_{i_{m}}s_{i_{m+1}}\cdots s_{i_{2m-2}}(q_{m})
\end{eqnarray}

Now we prove the case (1), ${\cal E}(C) = s_0s_{2m-1}C$.
In this case for $p_j$ ($j=1,2,\ldots, m-1$) we have
\begin{eqnarray*}
\lefteqn{s_0(s_{i_{m}}s_{i_{m+1}}\cdots s_{i_{2m-2}})
	{\cal E}(C)^{m}(p_j)}\\
&=&	s_0(s_{i_{m}}s_{i_{m+1}}\cdots s_{i_{2m-2}})
	(s_0s_{2m-1}C)^{m}(p_j)\\
&=&	s_0(s_{i_{m}}s_{i_{m+1}}\cdots s_{i_{2m-2}})
	(s_0s_{2m-1}C)^{m-1}(s_0s_{2m-1}C)(p_j)\\
&=&	s_0(s_{i_{m}}s_{i_{m+1}}\cdots s_{i_{2m-2}})
	(s_0s_{2m-1}C)^{m-1}(s_0s_{2m-1})(p_{j+1})\\
&=&	s_0(s_{i_{m}}s_{i_{m+1}}\cdots s_{i_{2m-2}})
	(s_0s_{2m-1}C)^{m-1}(p_{j+1})\\
&=&	\cdots\\
&=&	s_0(s_{i_{m}}s_{i_{m+1}}\cdots s_{i_{2m-2}})
	(s_0s_{2m-1}C)^{j+1}(p_{m-1})\\
&=&	s_0(s_{i_{m}}s_{i_{m+1}}\cdots s_{i_{2m-2}})
	(s_0s_{2m-1}C)^{j}(s_0s_{2m-1}C)(p_{m-1})\\
&=&	s_0(s_{i_{m}}s_{i_{m+1}}\cdots s_{i_{2m-2}})
	(s_0s_{2m-1}C)^{j}(s_0s_{2m-1})(q_1)\\
&=&	s_0(s_{i_{m}}s_{i_{m+1}}\cdots s_{i_{2m-2}})
	(s_0s_{2m-1}C)^{j}s_0(2m)\quad (\because q_1 = 2m-1)\\
&=&	s_0(s_{i_{m}}s_{i_{m+1}}\cdots s_{i_{2m-2}})
	(s_0s_{2m-1}C)^{j-1}(s_0s_{2m-1}C)(2m)\\
&=&	s_0(s_{i_{m}}s_{i_{m+1}}\cdots s_{i_{2m-2}})
	(s_0s_{2m-1}C)^{j-1}(s_0s_{2m-1})(2m)\\
&=&	s_0(s_{i_{m}}s_{i_{m+1}}\cdots s_{i_{2m-2}})
	(s_0s_{2m-1}C)^{j-1}s_0(q_1)\quad (\because q_1 = 2m-1)\\
&=&	s_0(s_{i_{m}}s_{i_{m+1}}\cdots s_{i_{2m-2}})
	(s_0s_{2m-1}C)^{j-1}(q_1)\\
&=&	\cdots\\
&=&	s_0(s_{i_{m}}s_{i_{m+1}}\cdots s_{i_{2m-2}})
	(q_j)\\
&=&	s_0(2m-p_j)\quad (\because \eqref{eq:Cpj})\\
&=&	2m-p_j\quad (\because p_j<2m-1)
\end{eqnarray*}
and for $q_j$ ($j = 2, 3, \ldots, m$) we have
\begin{eqnarray*}
\lefteqn{s_0(s_{i_{m}}s_{i_{m+1}}\cdots s_{i_{2m-2}})
	{\cal E}(C)^{m}(q_j)}\\
&=&	s_0(s_{i_{m}}s_{i_{m+1}}\cdots s_{i_{2m-2}})
	(s_0s_{2m-1}C)^{m}(q_j)\\
&=&	s_0(s_{i_{m}}s_{i_{m+1}}\cdots s_{i_{2m-2}})
	(s_0s_{2m-1}C)^{m-1}(q_{j+1})\\
&=&	s_0(s_{i_{m}}s_{i_{m+1}}\cdots s_{i_{2m-2}})
	(s_0s_{2m-1}C)^{j}(q_{m})\\
&=&	s_0(s_{i_{m}}s_{i_{m+1}}\cdots s_{i_{2m-2}})
	(s_0s_{2m-1}C)^{j-1}(s_0s_{2m+1})(p_1)\\
&=&	s_0(s_{i_{m}}s_{i_{m+1}}\cdots s_{i_{2m-2}})
	(s_0s_{2m-1}C)^{j-1}(0)\quad (\because p_1 = 1)\\
&=&	s_0(s_{i_{m}}s_{i_{m+1}}\cdots s_{i_{2m-2}})
	(s_0s_{2m-1}C)^{j-2}(p_1)\quad (\because p_1 = 1)\\
&=&	s_0(s_{i_{m}}s_{i_{m+1}}\cdots s_{i_n})(p_{j-1})\\
&=&	s_0(2m-q_j)\quad (\because \eqref{eq:Cqj})\\
&=&	2m-q_j\quad (\because q_j<q_1 = 2m-1).
\end{eqnarray*}
For $q_1 = 2m-1$ we have
\begin{eqnarray*}
\lefteqn{s_0(s_{i_{m}}s_{i_{m+1}}\cdots s_{i_{2m-2}})
	{\cal E}(C)^{m}(q_1)}\\
&=&	s_0(s_{i_{m}}s_{i_{m+1}}\cdots s_{i_{2m-2}})
	(s_0s_{2m-1}C)^{m}(q_1)\\
&=&	s_0(s_{i_{m}}s_{i_{m+1}}\cdots s_{i_{2m-2}})
	(s_0s_{2m-1}C)(q_{m})\\
&=&	s_0(s_{i_{m}}s_{i_{m+1}}\cdots s_{i_{2m-2}})
	(s_0s_{2m-1})(p_1)\\
&=&	s_0(s_{i_{m}}s_{i_{m+1}}\cdots s_{i_{2m-2}})
	(0)\quad (\because p_1 = 1)\\
&=&	1.
\end{eqnarray*}
Further we have
\begin{eqnarray*}
\lefteqn{s_0(s_{i_{m}}s_{i_{m+1}}\cdots s_{i_{2m-2}})
	{\cal E}(C)^{m}(0)}\\
&=&	s_0(s_{i_{m}}s_{i_{m+1}}\cdots s_{i_{2m-2}})
	(s_0s_{2m-1}C)^{m}(0)\\
&=&	s_0(s_{i_{m}}s_{i_{m+1}}\cdots s_{i_{2m-2}})
	(s_0s_{2m-1}C)^{m-1}(p_1)\quad (\because p_1 = 1)\\
&=&	s_0(s_{i_{m}}s_{i_{m+1}}\cdots s_{i_{2m-2}})
	(s_0s_{2m-1}C)(p_{m-1})\\
&=&	s_0(s_{i_{m}}s_{i_{m+1}}\cdots s_{i_{2m-2}})
	(s_0s_{2m-1})(q_1)\\
&=&	s_0(s_{i_{m}}s_{i_{m+1}}\cdots s_{i_{2m-2}})
	(2m)\quad (\because q_1 = 2m-1)\\
&=&	2m\ .
\end{eqnarray*}
Finally we have
$s_0(s_{i_{m}}s_{i_{m+1}}\cdots s_{i_{2m-2}}){\cal E}(C)^{m}(2m) = 0$,
since $s_i$ ($i=0,1,\ldots, 2m$) is bijection.
Thus, we found that $s_0(s_{i_{m}}\cdots s_{i_{2m-2}}){\cal E}(C)^{m}$
is the longest.
In order to continue the induction,
we further have to show that
\[
	{\cal E}(C)^{m+1/2}
 = s_0(s_{i_{m}}\cdots s_{i_{2m-2}}){\cal E}(C)^{m}.
\]
In other words,
we have to show that ${\cal E}(C)$
has an expression ${\cal E}(C) = w_1w_2$
such that $w_2 = s_0(s_{i_{m}}\cdots s_{i_{2m-2}})$
and $\ell(w_1) = \ell(w_2) = m$.
Since we already know that ${\cal E}(C)$ has an expression
$s_0s_{2m-1}(s_{i_1}\cdots s_{i_{m-1}})(s_{i_{m}}\cdots s_{i_{2m-2}})$,
we have only to show that
\begin{eqnarray}\label{eq:moves0}
\lefteqn{s_0s_{2m-1}(s_{i_1}\cdots s_{i_{m-1}})
	(s_{i_{m}}\cdots s_{i_{2m-2}})
}\nonumber\\
&=&
s_{2m-1}(s_{i_1}\cdots s_{i_{m-1}})
s_0(s_{i_{m}}\cdots s_{i_{2m-2}}).
\end{eqnarray}
By the equation~\eqref{eq:Cq1} we find that
$s_{i_{m}}\cdots s_{i_{2m-2}}$ involves $s_1$.
This means $s_{2m-1}(s_{i_1}\cdots s_{i_{m-1}})$ does not involve $s_1$.
Hence we can move the $s_0$
in the left hand side of the equation~\eqref{eq:moves0}
rightward and we have the right hand side.

Next we prove the case (2), ${\cal E}(C) = Cs_0s_{2m-1}$.
Similarly to the case (1),
we can check the following by direct calculation.
For $p_j$ ($j = 2, 3, \ldots, m-1$) we have
\begin{equation*}
(s_{i_{m}}s_{i_{m+1}}\cdots s_{i_{2m-2}})s_0{\cal E}(C)^{m}(p_j)
=	2m-p_j
\end{equation*}
and
\begin{equation*}
(s_{i_{m}}s_{i_{m+1}}\cdots s_{i_{2m-2}})s_0{\cal E}(C)^{m}(p_1)\\
=	2m-p_1\ .
\end{equation*}
For $q_j$ ($j = 3, 4, \ldots, m$) we have
\begin{equation*}
(s_{i_{m}}s_{i_{m+1}}\cdots s_{i_{2m-2}})s_0{\cal E}(C)^{m}(q_j)
=	2m-q_j\ ,
\end{equation*}
\begin{equation*}
(s_{i_{m}}s_{i_{m+1}}\cdots s_{i_{2m-2}})s_0{\cal E}(C)^{m}(q_2)\\
=	2m-q_2
\end{equation*}
and
\begin{equation*}
(s_{i_{m}}s_{i_{m+1}}\cdots s_{i_{2m-2}})s_0{\cal E}(C)^{m}(q_1)\\
=	2m-q_1\ .
\end{equation*}
Further we have
\begin{equation*}
(s_{i_{m}}s_{i_{m+1}}\cdots s_{i_{2m-2}})s_0{\cal E}(C)^{m}(0)\\
=	2m
\end{equation*}
and
\begin{equation*}
(s_{i_{m}}s_{i_{m+1}}\cdots s_{i_{2m-2}})s_0
	({\cal E}(C))^{m}(2m)\\
= 0\ .
\end{equation*}
Thus, we found that $(s_{i_{m}}\cdots s_{i_{2m-2}})s_0{\cal E}(C)^{m}$
is the longest.
In order to continue the induction,
we have to check that ${\cal E}(C)$
can be written as ${\cal E}(C) = w_1w_2$
such that $w_2 = (s_{i_m}s_{i_{m+1}}\cdots s_{i_{2m-2}})s_0$
and $\ell(w_1) = \ell(w_2) = m$.
Since ${\cal E}(C) = Cs_0s_{2m-1}
= (s_{i_1}\cdots s_{i_{m-1}})(s_{i_m}\cdots s_{i_{2m-2}})s_0s_{2m-1}$,
we have to show that
\begin{eqnarray}\label{eq:moves2}
\lefteqn{(s_{i_1}\cdots s_{i_{m-1}})
	(s_{i_{m}}\cdots s_{i_{2m-2}})s_0s_{2m-1}
}\nonumber\\
&=&
(s_{i_1}\cdots s_{i_{m-1}})s_{2m-1}
(s_{i_{m}}\cdots s_{i_{2m-2}})s_0.
\end{eqnarray}
By the equation~\eqref{eq:Cpj},
$s_{i_m}\cdots s_{i_{2m-2}}$
maps $q_1 =2m-1$ to $2m-p_1 = 2m-1$.
This implies $s_{i_m}\cdots s_{i_{2m-2}}$
does not involve $s_{2m-2}$.
Hence we can move $s_{2m-1}$
in the left hand side of the equation~\eqref{eq:moves2} leftward.
Thus we obtain the right hand side of the equation.

The case (3), ${\cal E}(C) = s_0Cs_{2m-1}$
is also verified by direct calculation.
In this case $(s_{i_{m}}\cdots s_{i_{2m-2}})s_{2m-1}{\cal E}(C)^m$
is already desired form.
So we can continue the induction.

Thus we have completed the proof of the lemma.
\end{proof}
Finally, we can obtain the following theorem.
\begin{thm}
Let $C$ be a Coxeter element in $\mathfrak{S}_{2m-1}$.
If $C$ is admissible, then there exists an expression $w_1w_2$ of $C$
such that $C^{h/2}_{w_2} = w_2C^{m-1}$ affords the longest.
Conversely, if $C^{h/2}_{w_2} = w_2C^{m-1}$ is the longest
in $\mathfrak{S}_{2m-1}$ for an expression $w_1w_2$ of $C$,
then $C$ is admissible.
\end{thm}
\begin{proof}[{\sc Proof}]
By Remark~\ref{rem:admissible} and the previous lemma,
we find that the first statement of the theorem holds.
In order to prove the second statement
we have to show that
if $C\in\mathfrak{S}_{2m-1}$ is non-admissible Coxeter element
then $C^{h/2} = C^{h/2}_{w_2}$ in \eqref{eq:hpCox}
is not the longest word in $\mathfrak{S}_{2m-1}$
for any choice of $w_2$.

First we show that if $ht(C) = \pm 3$
then $C^{h/2}= C^{h/2}_{w_2}$ can not be the longest.
Suppose that $ht(C) = 3$.
By Remark~\ref{rem:cyclic}, the numbers of stanzas and co-stanzas of $C$
are $m-2$ and $m+1$ respectively,
and $C = (p_1, p_2, \ldots, p_{m-2}, q_1, q_2, \ldots, q_{m+1})$.
Hence $C^{m-1}$ becomes (in two-line form) as follows.
\[
	\begin{pmatrix}
		p_1&p_2&\cdots&p_{m-2}&q_1&q_2&q_3&\cdots&q_{m+1}\\
		q_2&q_3&\cdots&q_{m-1}&q_m&q_{m+1}&p_1&\cdots&p_{m-1}
	\end{pmatrix}\ .
\]
Note that $q_3$ goes to $p_1=1$ and $q_1= 2m-1$ goes to $q_m$.
Under the condition that
$w_2$ consists of $m-1$ distinct generators
from $\{s_1, \ldots, s_{2m-1}\}$,
in order that $w_2C^{m-1}$ is the longest,
$w_2$ has to map $1 = p_1 = C^{m-1}(q_3)$ to $2m - q_3$ and
$q_m = C^{m-1}(q_1) = C^{m-1}(2m-1)$ to 1.
Since $1=p_1<q_{m+1}<q_m$ and $2m-1= q_1>q_2>q_3$,
we have $q_m \geq 3$ and $2m - q_3 \geq 3$.
Hence $w_2$ has to satisfy $w_2(1)\geq 3$ and $w_2(q_m) = 1$($q_m\geq 3$).
The former implies $w_2$ involves a sequence $v_0s_2v_1s_1v_2$
and the latter implies it involves a sequence $v_0's_1v_1's_2v_2'$.
Here $v_i$s and $v_i'$s are sequences of the generators
which have neither $s_1$ nor $s_2$ and which may be empty.
This contradicts the assumption that $w_2$ is a distinct product
of the generators.
Hence if $ht(C) = 3$, then $C^{h/2} = C^{h/2}_{w_2}$ does not afford
the longest element for any $w_2$.
Similarly if $ht(C) = -3$, then $C^{h/2}$ does not afford
the longest either.

Next we show that
if $ht(C) =\pm 3$ then $\ell(C^{m-1})<2(m-1)^2$.
Suppose that $C$ is a Coxeter element in $\mathfrak{S}_{2m-1}$
and $ht(C)$ is 3 or $-3$.
For $C^{m-1}\in\mathfrak{S}_{2m-1}$
there uniquely exist $w'$ and $w''$ such that
$w''C^{m-1}$ and $C^{m-1}w'$ are the longest~\cite{H}.
Here we note that $\ell(w') = \ell(w'')\geq m-1$.
Hence for the $w'$ and $w''$
we have $C^{m-1}w'w''C^{m-1} = 1$.
On the other hand, we already have $C^h = C^{m-1}CC^{m-1} = 1$.
Hence we have $C = w'w''$.
If $\ell(w') = \ell(w'') = m-1$,
then $w'$ and $w''$ must be distinct products of the $m-1$ generators.
This contradicts the previous argument.
Hence we have $\ell(w'') > m-1$.
Since the length of the longest word in $\mathfrak{S}_{2m-1}$
is $(m-1)(2m-1)$, we have $\ell(C^{m-1})<2(m-1)^2$.

Finally we show that if $C$ is non-admissible
then $\ell(C^{m-1})<2(m-1)^2$.
By the previous argument if $ht(C) = \pm 3$,
then the claim holds.
By the definition of non-admissible,
every non-admissible Coxeter element is obtained
from iterative extensions of a Coxeter
element of height $\pm 3$.
So we have only to show the following:
\begin{quote}
For a Coxeter element $C\in\mathfrak{S}_{2m-1}$,
if $\ell(C^{m-1})<2(m-1)^2$,
then $\ell({\cal E}(C)^m) <2m^2$.
\end{quote}
Imagine the Amida diagrams of $C^{m-1}$
and ${\cal E}(C)^m$.
The latter diagram is obtained from the former
by adding $m$ horizontal segments which correspond to $s_0$
to the left and adding another $m$ horizontal segments
which correspond to $s_{2m-1}$ to the right
and $2m-2$ horizontal segments which correspond to $C$
to the above.
Hence we have
$\ell({\cal E}(C)^m) \leq \ell(C^{m-1}) + 4m - 2<2m^2$
and we completes the proof of the theorem.
\end{proof}


\end{document}